\documentclass{lmcs}
\pdfoutput=1

\usepackage{lastpage}
\renewcommand{\dOi}{14(3:18)2018}
\lmcsheading{}{1--\pageref{LastPage}}{}{}%
{Mar.~02,~2018}{Sep.~11,~2018}{}


\usepackage[utf8]{inputenc}
\usepackage[T1]{fontenc}
\usepackage{lmodern}

\usepackage[pdf,all,2cell]{xy}\SelectTips{cm}{10}\SilentMatrices\UseAllTwocells
\usepackage{xspace}
\usepackage{etoolbox}
\usepackage{xifthen} 
\usepackage[multiple]{footmisc}

\theoremstyle{plain}

\newtheorem{lemma}[thm]{Lemma}
\newtheorem{corollary}[thm]{Corollary}
\newtheorem{problem}[thm]{Problem}
\theoremstyle{definition}
\newtheorem{definition}[thm]{Definition}
\newtheorem{example}[thm]{Example}
\theoremstyle{remark}
\newtheorem*{remark}{Remark}
\theoremstyle{definition}
\newtheorem{constrInternal}[thm]{Construction}

\newenvironment{construction}[2][]
  {\pushQED{\qed}\begin{constrInternal}[{for Problem~\ref{#2}\ifthenelse{\isempty{#1}}{}{; #1}}]}
  {\popQED\end{constrInternal}}

\newcounter{saveenumi}
\newcommand{\saveitem}{\setcounter{saveenumi}{\value{enumi}}}
\newcommand{\restoreitem}{\setcounter{enumi}{\value{saveenumi}}}

\newcommand{\grants}[1]{\newcommand{\thegrants}{#1}}

\title{Categorical structures for type theory in univalent foundations}

\grants{%
  This material is based upon work supported by the National Science Foundation under agreement Nos. DMS-1128155 and CMU 1150129-338510.
  Any opinions, findings and conclusions or recommendations expressed in this material are those of the author(s) and do not necessarily reflect the views of the National Science Foundation.
  This material is based on research sponsored by The United States Air Force Research Laboratory under agreements number FA9550-15-1-0053 and FA9550-17-1-0363. The US Government is authorized to reproduce and distribute reprints for Governmental purposes notwithstanding any copyright notation thereon.
The views and conclusions contained herein are those of the author and should not be interpreted as necessarily representing the official policies or endorsements, either expressed or implied, of the United States Air Force Research Laboratory, the U.S. Government or Carnegie Mellon University.
  This work has partly been funded by the CoqHoTT ERC Grant 637339,
  and by the Swedish Research Council (VR) Grant 2015-03835 \emph{Constructive and category-theoretic foundations of mathematics}.
}

\newcommand{\C}{{\mathcal{C}}}
\newcommand{\D}{{\mathcal{D}}}
\newcommand{\ES}{\mathcal{S}}
\newcommand{\A}{{\mathcal{A}}}
\newcommand{\X}{{\mathcal{X}}}
\newcommand{\constfont}[1]{\ensuremath{\mathsf{#1}}}
\newcommand{\Ty}{\constfont{Ty}}
\newcommand{\Tm}{\constfont{Tm}}
\newcommand{\op}{^{\mathrm{op}}}
\newcommand{\set}{\ensuremath{\constfont{Set}}\xspace}
\newcommand{\Fam}{\ensuremath{\constfont{Fam}}\xspace}
\newcommand{\p}{\constfont{p}}
\newcommand{\Tymap}[1]{\ensuremath{\Ty({#1})}}
\newcommand{\Tmmap}[1]{\ensuremath{\Tm({#1})}}
\newcommand{\compext}[2]{\ensuremath{{#1}.{#2}}}
\newcommand{\compextcompare}[1][]{\Delta_{#1}}
\newcommand{\deppr}[1]{\constfont{\pi}_{#1}}
\newcommand{\te}{\ensuremath{\constfont{te}}}
\newcommand{\yofunctor}[1][]{\constfont{y}_{#1}}
\newcommand{\yotranspose}[1]{\widehat{#1}}
\newcommand{\q}{\constfont{q}}
\newcommand{\qmap}[2]{\ensuremath{\q({#1},{#2})}}
\newcommand{\idtoiso}{\ensuremath{\mathsf{idtoiso}}\xspace}
\newcommand{\iso}{\cong} 
\newcommand{\Iso}{\constfont{Iso}} 
\newcommand{\mapfunc}[1]{\ensuremath{\mathsf{ap}_{#1}}\xspace} 
\newcommand{\compose}[2]{\ensuremath{{#1}\cdot{#2}}}
\newcommand{\reindex}[2]{\ensuremath{#1^*{#2}}}
\newcommand{\U}{\constfont{U}}
\newcommand{\tU}{\constfont{\tilde{U}}}
\newcommand{\Jpb}[2]{\ensuremath{{#1}.{#2}}}
\newcommand{\Jpr}[1]{\ensuremath{\p_{#1}}}
\newcommand{\JQ}[2]{\ensuremath{\constfont{Q}({#1},{#2})}}
\newcommand{\RC}[1]{\ensuremath{\constfont{RC}({#1})}}
\newcommand{\preShv}[1]{\ensuremath{\constfont{PreShv}({#1})}}

\newcommand{\stype}{\constfont{splty}}
\newcommand{\cwf}{\constfont{cwf}}
\newcommand{\rep}{\constfont{rep}}
\newcommand{\relu}{\constfont{relu}}
\newcommand{\relwku}{\constfont{relwku}}
\newcommand{\compat}{\constfont{compat}}
\newcommand{\qmor}{\constfont{qmor}}
\newcommand{\tmstr}{\constfont{tmstr}}
\newcommand{\objext}{\constfont{objext}}
\newcommand{\diag}[1]{\delta_{#1}}

\newcommand{\coqident}[1]{\nolinkurl{#1}} 

\newcommand{\UniMath}{\href{https://github.com/UniMath/UniMath}{\nolinkurl{UniMath}}\xspace}

\usepackage{pigpen}
\newcommand{\pb}{\ar@{}[dr]|<<{\text{\pigpenfont A}}}

\begin{document}

\author{Benedikt Ahrens} 
\address{School of Computer Science, University of Birmingham, United Kingdom}
\email{b.ahrens@cs.bham.ac.uk}
\author{Peter LeFanu Lumsdaine} 
\address{Department of Mathematics, Stockholm University, Sweden}
\email{p.l.lumsdaine@math.su.se}
\author{Vladimir Voevodsky}
\address{Institute for Advanced Study, Princeton, NJ, USA}
\email{vladimir@ias.edu}

\maketitle

\begin{abstract}
  In this paper, we analyze and compare three of the many algebraic structures that have been used for modeling dependent type theories:
  \emph{categories with families}, \emph{split type-categories}, and \emph{representable maps of presheaves}.
  We study these in univalent type theory, where the comparisons between them can be given more elementarily than in set-theoretic foundations.
  Specifically, we construct maps between the various types of structures, and show that assuming the Univalence axiom, some of the comparisons are equivalences.

  We then analyze how these structures transfer along (weak and strong) equivalences of categories, and, in particular,
  show how they descend from a category (not assumed univalent/saturated) to its Rezk completion.
  To this end, we introduce \emph{relative universes}, generalizing the preceding notions,
  and study the transfer of such relative universes along suitable structure.

  We work throughout in (intensional) dependent type theory; some results, but not all, assume the univalence axiom. All the material of this paper has been formalized in Coq, over the \UniMath library.

\end{abstract}

\setcounter{tocdepth}{1}
\tableofcontents
~ 

\section{Introduction}

Various kinds of categorical structures have been introduced in which to interpret type theories:
for instance, categories with families, C-systems, categories with attributes, and so on.
The aim of these structures is to encompass and abstract away the structural rules of
type theories---weakening, substitution, and so on---that are
independent of the specific logical type and term formers of the
type theory under consideration.

For a given kind of categorical structures and a given type theory one expects furthermore to be able to build, from the syntax,
an initial object of a category or 2-category whose objects are categorical structures of this kind
equipped with suitable extra operations ‘modeling’ the type and term formers of the
type theory under consideration.
This universal property then provides a means to construct interpretations of
the syntax, by assembling the objects of the desired interpretation into another instance of such a structure.

It is natural to ask if, and in what sense, these various
categorical structures are equivalent, or otherwise related.
The equivalences, differences, and
comparisons between them are often said to be well-known, but few
precise statements exist in the literature.

The goals of the present work are twofold. Firstly, to give some such
comparisons precisely and carefully.  And secondly, to illustrate how in
univalent foundations, such comparisons can be approached in a different
and arguably more straightforward fashion than in classical settings.

More specifically, the paper falls into two main parts:
\begin{itemize}
\item comparison between categories with families (CwF’s) and split type-categories;
\item interaction between CwF structures and Rezk completion of categories,
  and comparison between CwF structures and representable maps of presheaves, all via comparison with \emph{relative universes}
  (introduced in the present work).
\end{itemize}

Our constructions and equivalences may be summed up in the following diagram:
\[
   \xymatrix@C=2pc{
                 \stype(\C) \ar@{<->}[rr]^-{\simeq}
                 & & \cwf(\C) \ar[d] \ar@{<->}[r]^-{\simeq}
                 & \relu(\yofunctor[\C])  \ar[r] \ar[d]
                 & \relu(\yofunctor[\RC{\C}])\ar@{<->}[r]^-{\simeq} \ar@{<->}[d]^-{\simeq}
                 & \cwf(\RC{\C}) \ar@{<->}[d]^-{\simeq}
                 \\
                 & & \rep(\C) \ar@{<->}[r]^-{\simeq}
                 & \relwku(\yofunctor[\C]) \ar@{<->}[r]^-{\simeq}
                 & \relwku(\yofunctor[\RC{\C}])    \ar@{<->}[r]^-{\simeq}
                 &  \rep (\RC{\C})
   }
\]

All proofs and constructions of the article have been formalized in Coq, over the \UniMath library.
We take this as licence to err on the side of indulgence in focusing on the key ideas of constructions and
suppressing less-enlightening technical details, for which we refer readers to the formalization.
An overview the formalization is given in Section~\ref{sec:formalization}.
Throughout, we annotate results with their corresponding identifier in the formalization, in \coqident{teletype_font}.

\subsection*{Publication history}

A preliminary version of this article was published in the proceedings of Computer Science Logic 2017 \cite{alv1-csl}.
Changes from that version include:
\begin{itemize}
\item In Definition~\ref{def:rel.univ.struct}, of universes relative to a functor, we add here a restriction to fully faithful functors, and a discussion of the issues involved for general functors.
\item The addition of Example~\ref{ex:no-bijec-in-zf}, showing explicitly how some of the equivalences constructed assuming the Univalence Axiom may fail in its absence.
\item Various expository improvements.
\end{itemize}

Sadly the third-named author, Vladimir Voevodsky, passed away before the preparation of the present extended version.
The remaining authors are grateful to Daniel R.~Grayson, Vladimir's academic executor, for his advice and support in preparing the revisions for this version.

\section{Background}

\subsection{Categorical structures for type theory}\label{sec:cat_structures}

In this short section we briefly review some of the various categorical structures for dependent type theory introduced in the literature.
We do not aim to give a comprehensive survey of the field, but just to recall what pertains to the present paper.

The first introduced were Cartmell's \emph{contextual categories} \cite{Cartmell0,DBLP:journals/apal/Cartmell86}, since studied by Voevodsky under the name of \emph{C-systems} \cite{Csubsystems,Cfromauniverse}.

Pitts \cite[Def.~6.3.3]{PittsAM:catl} introduced \emph{split type-categories}, originally just as \emph{type-categories};
Van den Berg and Garner \cite[Def.~2.2.1]{vandenBerg:2012:TSM:2071368.2071371} (whom we follow here) later reorganized Pitts' definition to isolate the splitness conditions, and hence allow what they call \emph{(non-split) type-categories}.
These have also been studied as \emph{categories with attributes} by Hofmann \cite{Hofmann97syntaxand} and others.%
\footnote{Cartmell originally used the name \emph{categories with attributes} for a slightly different notion \cite[\textsection 3.2]{Cartmell0}.}%
\footnote{Here and elsewhere, when we consider two notions of structure as the same without further justification, we mean that they are extremely trivially the same: that is, just up to reordering of components, distribution of $\Pi$-types over $\Sigma$-types, and similar mathematically negligible differences.}

\emph{Categories with families} were defined by Dybjer \cite[Def.~1]{Dybjer} to make explicit the data of \emph{terms}, not taken as primary in the previous approaches.
A functional variant of Dybjer's definition, which we follow in the present paper, was suggested by Fiore \cite[Appendix]{fiore_cwf},
and studied under the name \emph{natural models} by Awodey \cite[Def.~1]{awodey_natural_published} (see also footnote to Def.~\ref{def:rep_map} below).

The notions of \emph{universes} and \emph{universe categories}, which we generalize in the present work to \emph{universes relative to a functor},
were introduced by Voevodsky in \cite[Def.~2.1]{Cfromauniverse}.

\subsection{The agnostic, univalent, and classical settings} \label{sec:background-type-theory}

The background setting of the present work is intensional type theory, assuming throughout:
$\Sigma$-types, with the strong $\eta$ rule;
identity types;
$\Pi$-types, also with $\eta$, and functional extensionality;
$\constfont{0}$, $\constfont{1}$, $\constfont{2}$, and $\constfont{N}$;
propositional truncation (axiomatized as in \cite[\textsection 3.7]{HoTTbook}); and
two universes closed under all these constructions.

All the above is \emph{agnostic} about equality on types---it is not assumed either to be univalent, or to be always a proposition---and hence is expected to be compatible with the interpretation of types as classical sets, as well as homotopical interpretations such as the simplicial set model.
In particular, our main definitions of categorical structures use only this background theory, and under the classical interpretation they become the established definitions from the literature.
Similarly, most of the comparison maps we construct rely only on this, so can be understood in the classical setting.

Other results, however---essentially, all non-trivial equivalences of types we prove---assume additionally the univalence axiom; these therefore hold only in the univalent setting, and are not compatible with the classical interpretation.

We mostly follow the type-theoretic vocabulary standardized in \cite{HoTTbook}.
A brief, but sufficient, overview is given in \cite{rezk_completion}, among other places.
By $\set$ we denote the category of h-sets of a fixed, but unspecified universe.

We depart from it (and other type-theoretic traditions) in using \emph{existence} for what is called \emph{mere existence} in \cite{HoTTbook} and \emph{weak existence} by Howard~\cite{Howard80}, since this is what corresponds to the standard mathematical usage of \emph{existence}.

\subsection{Comparing structures in the univalent setting}

Suppose one wishes to show that two kinds of mathematical object are equivalent---say, one-object groupoids and groups.
What precise statement should one aim for?

In the classical setting, the most obvious candidate---a bijection of the sets (or classes) of these objects---is not at all satisfactory.
On the one hand, the natural back-and-forth constructions may well fail to form a bijection.
On the other hand, the axiom of choice may imply that bijections exist even when the objects involved are quite unrelated.

To give a more meaningful statement, one may define suitable morphisms, corral the objects into two categories, and construct an equivalence of categories between these.
Sometimes, one needs to go further, and construct an equivalence of higher categories, or spaces.

In the univalent setting, however, life is simpler.
The most straightforward candidate, an equivalence of \emph{types} between the types of the two kinds of objects, is already quite satisfactory and meaningful, corresponding roughly to an equivalence between the \emph{groupoids} of such objects in the classical setting (or higher groupoids, etc.).

This is not to dismiss the value given by constructing (say) an equivalence of categories
However, defining the morphisms and so on is no longer \emph{required} in order to give a meaningful equivalence between the two kinds of objects.%

Another advantage of a comparison in terms of equivalence of types is its \emph{uniformity}.
Indeed, the one notion of equivalence of types can serve to compare objects that naturally form
the elements of sets, or the objects of categories, or bicategories, etc.

In the present paper, therefore, we take advantage of this: two of our main results are such equivalences, between the types of categories with families and of split type-categories, and between the types of representable maps of presheaves on a category $\C$, and of CwF structures on its Rezk completion.

(We have focused here on equivalences, but the principle applies equally for other comparisons: for instance, an injection of types carries more information than an injection of sets/classes, corresponding roughly to a full and faithful map of (possibly higher) groupoids.)

\subsection{Categories in the univalent setting}

The fundamentals of category theory were transferred to the univalent setting in \cite{rezk_completion}.
Two primary notions of category are defined, there called \emph{precategories} and \emph{categories}.
We change terminology, calling their precategories \emph{categories} (since it is this that becomes the traditional definition under the set interpretation), and their categories \emph{univalent categories}.

Specifically, a \emph{category} $\C$ (in our terminology) consists of:
\begin{itemize}
 \item a type $\C_0$, its \emph{objects};
 \item for each $a, b : \C_0$, a set $\C(a,b)$, the \emph{morphisms} or \emph{maps} from $a$ to $b$;
 \item together with identity and composition operations satisfying the usual axioms.
\end{itemize}
We emphasize that the hom-sets $\C(a,b)$ are required to be sets, but $\C_0$ is allowed to be an arbitrary type.

In any category $\C$, there is a canonical map from equalities of objects to isomorphisms, $\idtoiso_{a,b} : (a =_{\C_0} b) \to \Iso_{\C}(a,b)$.
We say that $\C$ is \emph{univalent} if for all $a, b : \C_0$, this map $\idtoiso_{a,b}$ is an equivalence: informally, if ‘equality of objects is isomorphism’.

A central example is the category $\set$ of sets (in some universe).
Univalence of this category follows directly from the univalence axiom for the corresponding universe.
It follows in turn that $\preShv{\C}$, the category of presheaves on a category $\C$, is also always univalent.
We write $\yofunctor[\C] : \C \to \preShv{\C}$ for the Yoneda embedding.

In properties of functors, we distinguish carefully between existence properties and chosen data.
We say a functor $F : \C \to \D$ is \emph{essentially surjective} if for each $d : \D_0$, there exist some $c : \C_0$ and isomorphism $i : F c \iso d$, and is \emph{split essentially surjective} if it is equipped with an operation giving, for each $d : \D_0$, some such $c$ and $i$.
A \emph{weak equivalence} is a functor that is full, faithful, and essentially surjective.

An important construction from \cite{rezk_completion} is the \emph{Rezk completion} $\RC{\C}$ of a category $\C$, the ‘free’ univalent category
on $\C$.
Precisely, $\RC{\C}$ is univalent, and there is a weak equivalence $\eta_{\C} : \C \to \RC{\C}$, enjoying the expected universal property:
any functor from $\C$ to a univalent category factorizes uniquely through $\eta_{\C}$.
 \[
      \begin{xy}
        \xymatrix{
                  \C \ar[d]_{\eta} \ar[rd]^{F} & \\
                  \RC{\C} \ar@{-->}[r]_{\exists!} & **[r] \D \text{ univalent}
         }
      \end{xy}
  \]
Note that the main definitions make sense in the agnostic background setting, and under the classical interpretation become the standard definitions; but the Rezk completion construction and the univalence of $\set$ and $\preShv{\C}$ rely additionally on the univalence axiom.

Indeed, \cite{rezk_completion} additionally assumes \emph{resizing axioms}, in order to show that $\RC{\C}$ is no larger than $\C$
That is, a priori, the type $\RC{\C}_0$ lives in a higher universe than the type $\C_0$.
However, for the present paper, this size issue is not a problem; so we do not need to assume resizing rules.

As usual, we will write $f : a \to b$ for $f : \C(a,b)$ in arbitrary categories, and will write $c : \C$ rather than $c : \C_0$.
We write composition in the ‘diagrammatic’ order; that is, the composite of $f : a \to b$ and $g : b \to c$ is denoted $\compose{f}{g} : a \to c$.

\section{Equivalence between categories with families and split type-categories}\label{sec:cwf-to-stc}

\subsection{CwF’s and type-categories}

For the most part, we take care to follow established definitions closely.
We depart from most literature however in one way: we do not take CwF's or type-categories (or other similar structures) to include a terminal object.
This does not interact in any way with the rest of the structure, so does not affect the equivalences we construct below.
We do this since both versions (with and without terminal object) seem useful for the development of the theory; and it is easier to equip objects with extra structure later than to remove it.

\begin{definition}[{cf. \cite[Appendix]{fiore_cwf}}; {\coqident{cwf_structure}}, {\coqident{cwf}}]\label{def:cwf.fiore}
  A \emph{category with families (à la Fiore)} consists of:
  \begin{enumerate} \setcounter{enumi}{-1}
  \item a category $\C$, together with
  \item presheaves $\Ty, \Tm : \C\op \to \set$; \label{item:cwf.presheaves}
  \item a natural transformation $\p : \Tm \to \Ty$; and
  \item  for each object $\Gamma : \C$ and $A : \Tymap{\Gamma}$, a representation of the fiber of $\p$ over $A$, i.e. \label{item:cwf.fiber-rep}
    \begin{enumerate}
    \item an object $\compext{\Gamma}{A} : \C$ and map $\deppr{A} : \compext{\Gamma}{A} \to \Gamma$, \label{item:cwf.first}
    \item an element $\te_A:\Tmmap{\compext{\Gamma}{A}}$, such that $\p_{\compext{\Gamma}{A}}(\te_A) = \reindex{\deppr{A}}{A} : \Tymap{\compext{\Gamma}{A}},$
    \item and such that the induced commutative square \label{item:cwf.last}
      \[
      \begin{xy}
        \xymatrix{
                **[l]\yofunctor(\compext{\Gamma}{A})  \ar[r]^-{\yotranspose{\te_A}} \ar[d]_{\yofunctor(\deppr{A})}  \pb & \Tm \ar[d]^{\p}\\
                **[l]\yofunctor(\Gamma) \ar[r]_{\yotranspose{A}}& \Ty
              }
      \end{xy}
      \]
      is a pullback.
      (Here e.g.\ $\yotranspose{A}$ denotes the Yoneda transpose of an element of a presheaf.)
    \end{enumerate}
  \end{enumerate}
  By a \emph{CwF structure} on a category $\C$, we mean the data of items \ref{item:cwf.presheaves}--\ref{item:cwf.fiber-rep} above.
\end{definition}

\begin{remark}
  This is a reformulation, due to Fiore, of Dybjer’s original definition of CwF’s \cite[Def.~1]{Dybjer}, replacing the single functor $\C \to \Fam$ by the map of presheaves $\p : \Tm \to \Ty$.
\end{remark}

\begin{definition}
  Let $\C$ be a CwF, $\Gamma : \C$ an object, and $e : A = B$ an equality of elements of $\Tymap{\Gamma}$.
  We write $\compextcompare[e] : \compext{\Gamma}{A} \iso \compext{\Gamma}{B}$ for the induced isomorphism $\idtoiso(\mapfunc{x\mapsto \compext{\Gamma}{x}}(e))$.
\end{definition}

Since $\Tymap{\Gamma}$ is a set, we will sometimes suppress $e$ and write just $\compextcompare[A,B]$.
We will also use this notation in other situations with a family $\Ty$ and operation $\Gamma, A \mapsto \compext{\Gamma}{A}$ as in a CwF.

\begin{definition}[{cf.\ \cite[Def.~6.3.3]{PittsAM:catl}},{\cite[Def.~2.2.1]{vandenBerg:2012:TSM:2071368.2071371}}; \coqident{typecat_structure}, \coqident{is_split_}\allowbreak\coqident{typecat}, \coqident{split_typecat_structure}]\label{def:split.type.cat}

   A \emph{type-category} consists of:
  \begin{enumerate} \setcounter{enumi}{-1}
  \item a category $\C$, together with
  \item for each object $\Gamma : \C$, a type $\Tymap{\Gamma}$,
  \item for each $\Gamma : \C$ and $A : \Tymap{\Gamma}$, an object $\compext{\Gamma}{A} : \C$,
  \item for each such $\Gamma$, $A$, a morphism $\deppr{A} : \compext{\Gamma}{A} \to \Gamma$,
  \item for each map $f : \Gamma' \to \Gamma$, a ‘reindexing’ function $\Tymap{\Gamma} \to \Tymap{\Gamma'}$, denoted $A \mapsto \reindex{f}{A} $,
  \item for each $\Gamma$, $A : \Tymap{\Gamma}$, and $f : \Gamma' \to \Gamma$, a morphism
         $\qmap{f}{A} : \compext{\Gamma'}{\reindex{f}{A}} \to \compext{\Gamma}{A}$,
  \item such that for each such $\Gamma$, $A$, $\Gamma'$, $f$, the following square commutes and is a pullback:
         \[
           \xymatrix@C=5pc{
                        \compext{\Gamma'}{\reindex{f}{A}} \ar[r]^{\qmap{f}{A}} \ar[d]_{\deppr{\reindex{f}{A}}} \pb &
                                                                        \compext{\Gamma}{A} \ar[d]^{\deppr{A}} \\
                        \Gamma' \ar[r]_{f}& \Gamma
           }
         \]
  \saveitem
  \end{enumerate}

  \noindent A type-category is \emph{split} if:
  \begin{enumerate} \restoreitem
   \item for each $\Gamma$, the type $\Tymap{\Gamma}$ is a set;
   \item for each $\Gamma$ and $A : \Tymap{\Gamma}$, we have equalities
     \begin{enumerate}
     \item $e : \reindex{1_{\Gamma}}{A} = A$, and
     \item $\qmap{1_{\Gamma}}{A} = \compextcompare[e]\ :\ \compext{\Gamma}{\reindex{1_\Gamma}{A}} \to \compext{\Gamma}{A}$; and
     \end{enumerate}
   \item for $f' : \Gamma'' \to \Gamma')$, $f : \Gamma' \to \Gamma$, and $A : \Tymap{\Gamma}$, we have equalities
     \begin{enumerate}
     \item $e' : \reindex{(\compose{f'}{f})}{A} = \reindex{f'}{\reindex{f}{A}}$, and
     \item $\qmap{\compose{f'}{f}}{A}
       = \compose{\compose{\compextcompare[e']}{\qmap{f'}{\reindex{f}{A}}}} {\qmap{f}{A}}
       \ :\ \compext{\Gamma''}{\reindex{(\compose{f'}{f})}{A}} \to \compext{\Gamma}{A}.$
     \end{enumerate}
  \end{enumerate}
 \end{definition}

 In the present work, we will only consider split type-categories.
 Non-split type-categories are however also of great importance, especially in the agnostic/univalent settings, since classical methods for constructing split ones may no longer work.

\subsection{Equivalence between CwF’s and split type-categories}

The main goal of this section is to construct an equivalence of types between the type of CwF’s and the type
of split type-categories.  In outline, we proceed as follows:

Firstly, we specialize to giving an equivalence between CwF structures and split type-category structures over a fixed base category $\C$.

Secondly, we further abstract out the shared part of these, decomposing them into
 \begin{itemize}
  \item \emph{object extension structures}, the shared structure common to CwF’s and split type-categories, and
  \item structures comprising the remaining data of CwF structures and split type-category structures, which we call \emph{term-structures} and \emph{$\q$-morphism structures} respectively.
 \end{itemize}

Finally, we give an equivalence between term-structures and $\q$-morphism structures over a given category and object extension structure.
We do this by defining a \emph{compatibility} relation between them, and showing that for each term-structure there exists a unique compatible $\q$-morphism structure, and vice versa.
\[ \xymatrix@C=0.5 pc@R=3ex{
   & & & & \displaystyle \sum_{X,Y,Z} \compat_X(Y,Z) \ar@{<->}[ld]_-{\simeq} \ar@{<->}[rd]^-{\simeq}
   \\
   \stype(\C) \ar@{<->}[rrr]^-{\simeq}
   & & & \displaystyle \sum_{X : \objext(\C)} \!\!\!\! \qmor(X) \ar[dr]
   &
   & \displaystyle \sum_{X : \objext(\C)} \!\!\!\! \tmstr(X) \ar[dl] \ar@{<->}[rrr]^-{\simeq}
   & & & \cwf(\C)
   \\
   & & & & \objext(\C)
} \]

 \begin{definition}[\coqident{obj_ext_structure}]
  A \emph{(split) object extension structure} on a category $\C$ consists of:
  \begin{enumerate}
   \item a functor $\Ty : \C\op \to \set$;
   \item for each $\Gamma : \C$ and $A : \Tymap{\Gamma}$,
      \begin{enumerate}
       \item an object $\compext{\Gamma}{A}$, and
       \item a \emph{projection} morphism $\deppr{A} : \C(\compext{\Gamma}{A}, \Gamma)$.
      \end{enumerate}
  \end{enumerate}
  In general, one might want to loosen the setness and functoriality conditions on $\Ty$, and hence distinguish the present definition as the split version of a more general notion.
  In this paper, however, we do not consider the non-split case, so take \emph{object extension structures} to mean the split ones throughout.
 \end{definition}

\begin{definition}[\coqident{term_fun_structure}]\label{def:cartesian.generator}
  Let $\C$ be a category equipped with an object extension structure $X = (\Ty,\compext{-}{-},\pi)$.
  A \emph{(functional) term-structure} over $X$ consists of:
 \begin{enumerate}
  \item a presheaf $\Tm : \C\op \to \set$, and natural transformation $\p : \Tm \to \Ty$;
  \item  for each object $\Gamma : \C$ and $A : \Tymap{\Gamma}$, an element $\te_A:\Tmmap{\compext{\Gamma}{A}}$, such that $(\compext{\Gamma}{A}, \deppr{A}, \te_A)$ form a representation of the fiber of $\p$ over $A$ as in item \ref{item:cwf.fiber-rep} of Def.~\ref{def:cwf.fiore}.
 \end{enumerate}
 One might say \emph{functional} to distinguish these from \emph{familial} term-structures, which would correspond analogously to CwF’s in the sense of Dybjer, with $\Tmmap{\Gamma}$ a family indexed by $\Tymap{\Gamma}$.
 For the present paper, however, we work only with the functional ones, so call these simply \emph{term-structures}.
\end{definition}

\begin{problem} \label{prob:weq-cwf-regrouped}
 Given a category $\C$, to construct an equivalence between CwF structures on $\C$ and pairs $(X,Y)$ of an object extension structure $X = (\Ty,\compext{-}{-},\pi)$ on $\C$ and a term-structure $Y$ over $X$.
\end{problem}

\begin{construction}[\coqident{weq_cwf_cwf'_structure}]{prob:weq-cwf-regrouped} \label{constr:weq-cwf-regrouped}
  Mathematically, this is essentially trivial, as is visibly evident from the definitions: just a matter of reordering and reassociating $\Sigma$-types, and distributing $\Pi$-types over $\Sigma$-types.
  It is perhaps worth noting however that this distributivity requires functional extensionality.
\end{construction}

\begin{definition}[\coqident{qq_morphism_structure}]\label{def:cartesian.q.mors}
  Let $\C$ be a category equipped with an object extension structure $X = (\Ty,\compext{-}{-},\pi)$.
  A \emph{(split) $\q$-morphism structure} over $X$ consists of:
  \begin{enumerate}
   \item for each $f : \Gamma' \to \Gamma$ and $A : \Tymap{\Gamma}$,
          a map $\qmap{f}{A} : \compext{\Gamma'}{\reindex{f}{A}} \to \compext{\Gamma}{A}$, such that following square commutes and is a pullback
          \[
           \xymatrix@C=5pc{
                        \compext{\Gamma'}{\reindex{f}{A}} \ar[r]^{\qmap{f}{A}} \ar[d]_{\deppr{\reindex{f}{A}}} \pb &
                                                                        \compext{\Gamma}{A} \ar[d]^{\deppr{A}} \\
                        \Gamma' \ar[r]_{f}& \Gamma
           }
         \]
         and such that
  \item for each $\Gamma$ and $A : \Tymap{\Gamma}$, $\qmap{1_{\Gamma}}{A} = \compextcompare[\reindex{1_\Gamma}{A},A]\ :\ \compext{\Gamma}{\reindex{1_\Gamma}{A}} \to \compext{\Gamma}{A}$; and
   \item for each $f' : \Gamma'' \to \Gamma'$, $f : \Gamma' \to \Gamma$, and $A : \Tymap{\Gamma}$,
       \[\qmap{\compose{f'}{f}}{A}
       = \compose{\compose%
           {\compextcompare[\reindex{(\compose{f'}{f})}{A},\reindex{f'}{\reindex{f}{A}}]}%
           {\qmap{f'}{\reindex{f}{A}}}}%
           {\qmap{f}{A}}
       \ :\  \compext{\Gamma''}{\reindex{(\compose{f'}{f})}{A}} \to \compext{\Gamma}{A}.\]
     \end{enumerate}

  Here the suppressed equalities on the $\compextcompare$ terms come from the functoriality axioms of $\Ty$.
\end{definition}

\begin{problem} \label{problem:weq-typecat-regrouped}
 Given a category $\C$, to construct an equivalence between split type-category structures on $\C$ and pairs $(X,Z)$ of an object extension structure $X$ on $\C$ and a $\q$-morphism structure $Z$ over $X$.
\end{problem}

\begin{construction}[\coqident{weq_standalone_to_regrouped}]{problem:weq-typecat-regrouped} \label{constr:weq-typecat-regrouped}
  Much like Construction \ref{constr:weq-cwf-regrouped}, simply a matter of wrangling $\Pi$- and $\Sigma$-types.
\end{construction}

For most of the remainder of this section, we fix a category $\C$ and object extension structure $X$ on $\C$.
We can now explicitly define constructions going back and forth between term-structures and $\q$-morphism structures over $X$, preparatory to showing that they form an equivalence.
Before we do so, we define the \emph{compatibility} relation between them.

\begin{definition}
  Let $Y=(\Tm,\p,\te)$ be a term-structure and $Z=(\q)$ a $\q$-morphism structure over $X$.
  Say $Y$ and $Z$ are \emph{compatible} if for all $f : \Gamma \to \Gamma'$ and $A : \Tymap{\Gamma}$,
  $\te_{\reindex{f}{A}} = \reindex{\qmap{f}{A}}{\te_A}$.  
\end{definition}

\begin{problem} \label{problem:term-to-compatible-qmor}
  Given a term-structure $Y$ over $X$, to construct a $\q$-morphism structure over $X$ compatible with $Y$.
\end{problem}

\begin{construction}[\coqident{compatible_qq_from_term}]{problem:term-to-compatible-qmor}
  Given $Y$ and $f : \Gamma \to \Gamma'$ and $A : \Tymap{\Gamma}$, the term-structure axioms give a pullback square
          \[ \xymatrix@C=5pc{
                \C(\compext{\Gamma'}{\reindex{f}{A}}, \compext{\Gamma}{A})  \ar[r]^{g \mapsto \reindex{g}{\te_A}} \ar[d]_{\compose{-}{\deppr{A}}} \pb
                    & \Tmmap{\compext{\Gamma'}{\reindex{f}{A}}}  \ar[d]^{\p_{\compext{\Gamma'}{\reindex{f}{A}}}} \\
                \C(\compext{\Gamma'}{\reindex{f}{A}}, \Gamma)  \ar[r]^{g \mapsto \reindex{g}{A}}
                    & \Tymap{\compext{\Gamma'}{\reindex{f}{A}}}
           } \]
  So we may take $\qmap{f}{A}$ to be the unique map $\compext{\Gamma'}{\reindex{f}{A}} \to \compext{\Gamma}{A}$ such that $\compose{\qmap{f}{A}}{\deppr{A}} = \compose{\deppr{\reindex{f}{A}}}{f}$ and $\reindex{\qmap{f}{A}}{\te_A} = \te_{\reindex{f}{A}}$.
  Verification that this forms a compatible $\q$-morphism structure is essentially routine calculation.
\end{construction}

\begin{problem} \label{problem:qmor-to-compatible-term}
 Given a $\q$-morphism structure $Y$ over $X$, to construct a term-structure over $X$ compatible with $Y$.
\end{problem}

\begin{construction}[\coqident{compatible_term_from_qq}]{problem:qmor-to-compatible-term}
  This construction is rather more involved; we only sketch it here, and refer to the formalization for full details.

  Briefly, $\Tmmap{\Gamma}$ is the set of pairs $(A,s)$, where $A : \Tymap{\Gamma}$, and $s : \Gamma \to \compext{\Gamma}{A}$ is a section of $\deppr{A}$.
  Its functorial action $\reindex{f}{}$ involves pulling back sections along the pullback squares given by the $\q$-morphism structure.
  Finally, the universal element $\te_A : \Tmmap{\Gamma}$ is the pair $(\reindex{\deppr{A}}{A}, \diag{\deppr{A}})$, where $\diag{\deppr{A}} : \compext{\Gamma}{A} \to \compext{\compext{\Gamma}{A}}{\reindex{\deppr{A}}{A}}$ is the diagonal map of the pullback square for $A$ and $\deppr{A}$ given by the $\q$-morphism structure.
\end{construction}

\begin{problem}
  \label{problem:weq-term-qmor}
   (Assuming Univalence.)
   To give an equivalence between term-structures and $\q$-morphism structures over $X$, whose underlying functions are as given in the two preceding constructions.

   From this equivalence, to derive an equivalence between the type of pairs $(X,Y)$ of an object extension structure and a term-structure,
   and of pairs $(X,Z)$ an object extension structure and a $\q$-morphism structure.
\end{problem}

\begin{construction}[\coqident{weq_cwf'_sty'}]{problem:weq-term-qmor}\label{construction:weq-term-qmor}
  As intimated above, we proceed by showing that for each term-structure, the compatible $\q$-morphism structure constructed above is in fact the \emph{unique} compatible such structure, and vice versa.

  This equivalence immediately induces an equivalence of pair types, which is the identity on the first component carrying the object extension structure.

  Note in particular that for this result---unlike in the constructions above---we rely essentially on the univalence axiom.
\end{construction}

\begin{problem}\label{problem:equiv_sty_cwf}
   (Assuming Univalence.)
   To construct an equivalence between CwF structures and split type-category structures on a category $\C$.
\end{problem}

\begin{construction}[\coqident{weq_sty_cwf}]{problem:equiv_sty_cwf}
  By composing the equivalences of Constructions~\ref{constr:weq-cwf-regrouped}, \ref{constr:weq-typecat-regrouped}, and \ref{construction:weq-term-qmor}:
  \[
    \stype(\C) \simeq  \sum_{X : \objext(\C)} \!\!\!\! \qmor(X) \simeq  \sum_{X
      : \objext(\C)} \!\!\!\! \tmstr(X) \simeq \cwf(\C) \enspace .
    \tag*{\qedhere}
  \]
\end{construction}

The back-and-forth constructions above between CwF structures and split type-category structures (though not their compatibility) are also sketched in \cite[\textsection \textsection 3.1, 3.2]{Hofmann97syntaxand} (note that Hofmann’s \emph{categories with attributes} are what we call \emph{split type-categories}).
Hofmann also observes that---in our terminology---going from $\q$-morphism structures to term-structures and back yields the original $\q$-morphism structure.
However, working in set-theoretic foundations, the same is not true for the other direction---the original term presheaf is not recovered up to equality (see Example~\ref{ex:no-bijec-in-zf} below), but only up to isomorphism.
This where the univalence axiom buys us a little extra:
it allows us to conclude, from that isomorphism, that the recovered term presheaf really is equal to the original. Hence we obtain an equivalence of types between CwF and split type-category structures.

(Note however that our proof that the maps form an equivalence is slightly less direct,
to avoid the difficult direct construction of an identity between two term-structures.)

\begin{example} \label{ex:no-bijec-in-zf}
   Let $\C$ be the terminal category $\{ \bullet \}$, and $\Ty : \C\op \to \set$ be the terminal functor, $\Ty(\bullet) := \{ \star \}$.
   There is a unique object extension structure extending $(\C,\Ty)$.

   A term-structure on this amounts to a set $\Tmmap{\bullet}$ and element $\te_\star \in \Tmmap{\bullet}$ such that
   \[
       \begin{xy}
        \xymatrix{
                           \C(\bullet,\bullet) \ar[r]^{\yotranspose{\te_\star}} \ar[d]   & \Tmmap{\bullet} \ar[d] \\
                           \C(\bullet,\bullet)  \ar[r] & \Tymap{\bullet}
        }
       \end{xy}
   \]
   is a pullback square.
   So $\Ty(\bullet)$ must be a singleton set, and given that, the rest of the data is uniquely determined.
   
   On the other hand, to give a $\q$-morphism structure on $(\C, \Ty)$, there is no choice at all:
   all the required data consists of maps in $\C$, so is uniquely determined.

   In other words, there is exactly one $\q$-morphism structure on $(\C, \Ty)$, while term-structures on $(\C, \Ty)$ correspond precisely to singletons in $\set$.
   
   In particular, in ZF(C) or similar set-theoretic settings, with the universe taken to be some Grothendieck universe, there are many singletons in $\set$.
   So there is no bijection between term-structures and $\q$-morphism structures on $(\C,\Ty)$: Problem~\ref{problem:weq-term-qmor} is insoluble.
\end{example}

In the absence of the univalence axiom, therefore, the comparison maps between term-structures and $\q$-morphism structures may not form an equivalence of types.
However, after defining suitable notions of morphisms of those structures, they can be shown to underlie an equivalence of \emph{categories}---which, from a classical viewpoint, is how one would expect to have to compare them.
We have also formalized this equivalence of categories, and will report on it in a forthcoming article.

As mentioned in this introduction, the equivalence (in some sense) of split type-categories and categories with families is generally considered well-known in folklore; however, we are unaware of any proof or precise statement in the literature, beyond a sketch of the back-and-forth constructions in \cite[\textsection 3.2]{Hofmann97syntaxand}.
Equivalences between several related notions are spelled out carefully in \cite[Ch.~1]{blanco}, including split type-categories (there called categories with attributes) but not including CwF’s.

\section{Relative universes and a transfer construction}

In this section we introduce the notion of \emph{universes relative to a functor $J : \C \to \D$}.
This notion generalizes the universes studied in \cite[Def.~2.1]{Cfromauniverse}, and is inspired by
the generalization of monads to relative monads \cite{DBLP:journals/corr/AltenkirchCU14}.

\subsection{Relative universes and weak universes}

 \begin{definition}[\coqident{fpullback}]\label{def:rel.univ.comp.struct}
  Let $J : \C\to\D$ be a functor, and  $\p : \tU \to \U$ a morphism of $\D$.

  Given an object $X$ of $\C$ and morphism $f: J(X) \to \U$ in $\D$,
  a \emph{$J$-pullback of $\p$ along $f$} consists of
  an object $X'$ of $\C$
  and morphisms $p' : X' \to X$ and $Q : J(X') \to \tU$,
  such that the following square commutes and is a pullback:
         \[
          \xymatrix{
                **[l] J(X') \ar[r]^-{Q} \ar[d]_{J(p')} \pb & \tU \ar[d]^{\p}\\
                **[l] J(X) \ar[r]_{f}& \U
                }
  \]
\end{definition}

\begin{definition}[\coqident{relative_universe}, \coqident{weak_relative_universe}]\label{def:rel.univ.struct}
  Let $J : \C \to \D$ be a fully faithful functor.

  A \emph{$J$-universe structure} on a map $\p : \tU \to \U$ of $\D$ is a function giving,
  for each object $X$ in $\C$ and map $f: J(X) \to \U$, a $J$-pullback $(\Jpb{X}{f}, \Jpr{f}, \JQ{X}{f})$ of $\p$ along $f$.
  A \emph{universe relative to $J$}, or briefly a \emph{$J$-universe}, is a map $\p$ equipped with a $J$-universe structure.

  A \emph{weak universe relative to $J$} is a map $\p : \tU \to \U$ such that for all suitable $X$, $f$, there exists some $J$-pullback of $\p$ along $f$.
\end{definition}

A \emph{universe} in $\C$, as defined in \cite{Cfromauniverse}, is exactly a universe relative to the identity functor $\constfont{Id}_{\C}$.
We will see in Section~\ref{sec:cwf_rezk} that universes relative to the Yoneda embedding $\yofunctor[\C] : \C \to \preShv{\C}$
correspond precisely to CwF structures on $\C$.

\begin{remark}
 In the conference version \cite{alv1-csl} of this article, the condition in 
 Definition~\ref{def:rel.univ.struct} on $J : \C \to \D$ being fully faithful was not imposed.
 
 This condition
 entails that the map $(X,f) \mapsto \Jpb{X}{f}$ can be equipped with a functorial action in a suitable sense, and that 
 the map $(X,f) \mapsto \Jpr{f}$ is natural with respect to this functorial action.
 More specifically, the type of this extra ``functoriality data'' on a $J$-universe $\p$ as above is a proposition when $J$ is faithful
 (\coqident{isaprop_functorial_structure_relu}), 
 and is contractible when $J$ is fully faithful
 (\coqident{iscontr_functorial_structure_relu}).
 
 If the restriction to fully faithful $J$ is dropped, 
 then the definition of $J$-universes should probably be expanded to include the functoriality data.
 
 However, our two motivating examples, $J=\constfont{Id}_{\C}$ and $J=\yofunctor[\C]$, are both fully faithful. 
 For the sake of simplicity, therefore, we give only this ``naïve'' definition of relative universes in the present work.
\end{remark}

\begin{lemma}[\coqident{isaprop_rel_universe_structure}]\label{lem:pb.univ.isaprop}
  Suppose $J : \C \to \D$ is full and faithful, and $\C$ is univalent.
  Then for any morphism $\p : \tU \to \U$ in $\D$, and object $X : \C$ and $f: J(X) \to \U$,
  the type of $J$-pullbacks of $\p$ along $f$ is a proposition.
  Similarly, for any such $\p$, the type of $J$-universe structures on $\p$ is a proposition.
\end{lemma}

\begin{proof}
  The first statement is similar to the argument that pullbacks in univalent categories are unique.
  The second follows directly from the first.
\end{proof}

\begin{corollary}[\coqident{weq_relative_universe_weak_relative_universe}] \label{cor:weq.relu.relwku}
  When $J : \C \to \D$ is fully faithful and $\C$ is univalent, the forgetful function from universes to weak universes relative to $J$ is an equivalence.
\end{corollary}

\subsection{Transfer constructions}

We give three constructions for transferring (weak) relative universes from one functor to another.
The first, with all data assumed to be given explicitly, is the most straightforward.
However, it does not suffice for transfers to the Rezk completion, since the embedding $\eta_{\C} : \C \to \RC{\C}$ is only a weak equivalence, not in general \emph{split} essentially surjective.
The second and third constructions are therefore adaptations of the first to require only essential surjectivity.

\begin{problem}\label{problem:universe.transfer.concrete}
  Given a square of functors commuting up to natural isomorphism, as in
  \[
      \begin{xy}
        \xymatrix{
                 \C \drtwocell<\omit>{\alpha} \ar[r]^{J} \ar[d]_{R}& \D \ar[d]^{S}\\
                 \C' \ar[r]_{J'} & \D',
        }
      \end{xy}
  \]
  such that
  \begin{itemize}
     \item $J$ and $J'$ are fully faithful,
     \item $S$ preserves pullbacks,
     \item $S$ is split full, and
     \item $R$ is split essentially surjective,
   \end{itemize}
   and given a $J$-relative universe structure on a map $\p : \tU \to \U$ in $\D$,
   to construct a $J'$-universe structure on $S(\p)$ in $\D'$.
\end{problem}

\begin{construction}[\coqident{rel_universe_structure_induced_with_ess_split}]{problem:universe.transfer.concrete} \label{constr:universe.transfer.concrete}
  Given $X$ in $\C'$ and $f : J'(X) \to S\U$, we need to construct a $J'$-pullback of $S(\p)$ along $f$.

  Split essential surjectivity of $R$ gives some $\bar{X} : \C$ and isomorphism $i : R(\bar{X}) \iso X$.
  We therefore have $\compose{\alpha_{\bar{X}}}{\compose{J'i}{f}} : SJ\bar{X} \to S\U$; so split fullness of $S$ gives us some $\bar{f} : J\bar{X} \to \U$ with $S \bar{f} = \compose{\alpha_{\bar{X}}}{\compose{J'i}{f}}$.
  Taking the given $J$-pullback of $\p$ along $\bar{f}$ and mapping it forward under $S$, we get a pullback square in $\D'$:
  \[
    \xymatrix@C=3.5pc{
      SJ(\Jpb{\bar{X}}{\bar{f}}) \ar[rrr]^{S(\JQ{\bar{X}}{\bar{f}})} \ar[d]_{SJ\Jpr{\bar{f}}} \pb
      & & & S \tU \ar[d]^{S \p}
      \\
      S J \bar{X} \ar[r]^{\alpha_{\bar{X}}}
      & J' R \bar{X} \ar[r]^{J'i}
      & J' X \ar[r]^f
      & S \U
    }
  \]

  The maps $\compose{R(\Jpr{\bar{f}})}{i} : R(\Jpb{\bar{X}}{\bar{f}}) \to X$ and $\compose{\alpha_{\Jpb{\bar{X}}{\bar{f}}}^{-1}}{S(\JQ{\bar{X}}{\bar{f}})} : J'R(\Jpb{\bar{X}}{\bar{f}}) \to S \tU$ now give the desired $J'$-pullback of $S\p$ along $f$, since the square they give---the right-hand square below---is isomorphic to the pullback obtained above.
  %
  \[ \begin{gathered}[b] \xymatrix@C=3pc{
      SJ(\Jpb{\bar{X}}{\bar{f}}) \ar[r]^{\alpha_{\Jpb{\bar{X}}{\bar{f}}}}_{\iso} \ar[d]_{SJ\Jpr{\bar{f}}}
      & J'R(\Jpb{\bar{X}}{\bar{f}}) \ar[r]^{1}_{\iso} \ar[d]^{J'R\Jpr{\bar{f}}}
      & J'R(\Jpb{\bar{X}}{\bar{f}}) \ar[rr]^{\compose{\alpha_{\Jpb{\bar{X}}{\bar{f}}}^{-1}}{S(\JQ{\bar{X}}{\bar{f}})}} \ar[d]^{J' (\compose{R(\Jpr{\bar{f}})}{i})} \pb
      & & S \tU \ar[d]^{S \p}
      \\
      S J \bar{X} \ar[r]^{\alpha_{\bar{X}}}_{\iso}
      & J' R \bar{X} \ar[r]^{J'i}_{\iso}
      & J' X \ar[rr]^f
      & & S \U
  } \\[-\dp\strutbox] \end{gathered} \qedhere \]
\end{construction}

\begin{problem}\label{problem:universe.struct.on}
  Given a square of functors as in Problem~\ref{problem:universe.transfer.concrete}, such that
  \begin{itemize}
     \item $J$ and $J'$ are fully faithful,
     \item $S$ preserves pullbacks,
     \item $S$ is full,
     \item $R$ is essentially surjective,
     \item $\C'$ is univalent,
   \end{itemize}
   and given a $J$-relative universe structure on a map $\p : \tU \to \U$ in $\D$,
   to construct a $J'$-universe structure on $S(\p)$ in $\D'$.
\end{problem}

\begin{construction}[\coqident{rel_universe_structure_induced_with_ess_surj}]{problem:universe.struct.on}\label{constr:universe.struct.on}
  Mostly the same as Construction~\ref{constr:universe.transfer.concrete}.
  The only problematic steps are finding $(\bar{X},i)$ and $\bar{f}$ as above, since the hypotheses which provided them there have now been weakened to existence properties.
  However, our new hypotheses that $\C'$ is univalent and $J'$ fully faithful allow us to apply Lemma~\ref{lem:pb.univ.isaprop} to see that the goal of a $J'$-pullback is a proposition; so the existence properties do in fact suffice.
\end{construction}

\begin{lemma}[\coqident{is_universe_transfer}]\label{lemma:transfer_relwku}
  Suppose given a square of functors as in Problem~\ref{problem:universe.transfer.concrete}, such that
  \begin{itemize}
     \item $J$ and $J'$ are fully faithful,
     \item $S$ preserves pullbacks,
     \item $S$ is full,
     \item $R$ is essentially surjective.
   \end{itemize}
   If $\p : \tU \to \U$ in $\D$ is a weak $J$-relative universe in $\C$,
   then $S(\p)$ is a weak $J'$-relative universe in $\D'$.
\end{lemma}

\begin{proof}
  Again, mostly the same as Construction~\ref{constr:universe.transfer.concrete}.
  Now our goal is just to show, for each suitable $X$, $f$, that there exists some $J'$-pullback of $S\p$ along $f$.
  This is by construction a proposition; so, again, we can obtain objects from our existence hypotheses whenever required.
\end{proof}

\begin{lemma}[\coqident{isweq_is_universe_transfer}]\label{lemma:transfer_relwku_equiv}
  Given functors satisfying the hypotheses of Lemma~\ref{lemma:transfer_relwku}, if additionally $R$ is full and $S$ is faithful,
  then a map $\p : \tU \to \U$ in $\D$ is a weak $J$-relative universe if and only if $S(\p)$ is a weak $J'$-relative universe.
\end{lemma}

\begin{proof}
  One implication is exactly Lemma~\ref{lemma:transfer_relwku}.  For the other, assume that $S(\p)$ is a weak $J'$-relative universe; we must show that $\p$ is a weak $J$-relative universe.

  Given $X : \C$ and $f : J(X) \to \U$, we need to show there exists a $J$-pullback for $\p$ along $f$.
  Since the goal is just existence, we can take witnesses for existence hypotheses as needed.
  In particular, we can take some $X' : \C'$, $p' : X' \to RX$ and $Q : J'X' \to S\tU$ forming a $J$-pullback for $S\p$ along $\compose{\alpha^{-1}_{X}}{Sf} : J'RX \to S\U$.

  By essential surjectivity of  $R$, we can find some $\bar{X}' : \C$ and isomorphism $i : R\bar{X}' \iso X'$;
  and similarly by its fullness, we can take some arrow $\bar{p}' : \bar{X}' \to X$ such that $R(\bar{p}') = \compose{i}{p'}$.
  Finally, fullness of $S$ gives a map $\bar{Q} : J \bar{X}' \to \tU$ with $S\bar{Q} = \compose{\alpha_{\bar{X}'}}{\compose{J'i}{Q}}$.

  But now $(\bar{X}',\bar{p}',\bar{Q})$ form a $J$-pullback for $\p$ along  $f$, since under $S$, the square they give becomes isomorphic to the original $J'$-pullback square of $(X',p',Q)$, and $S$ reflects limits, as it is full and faithful.
  \[ \begin{gathered}[b] \xymatrix@C=3pc{
      SJ\bar{X}' \ar[r]^{\alpha_{\bar{X}'}}_{\iso} \ar[d]_{SJ{\bar{p}'}}
      & J'R\bar{X}' \ar[r]^{J'i}_{\iso} \ar[d]_{J'R\bar{p}'}
      & J'X' \ar[r]^Q \ar[d]^{J' p'} \pb
      & S \tU \ar[d]^{S \p}
      \\
      S J \bar{X} \ar[r]^{\alpha_{X}}_{\iso}
      & J' R \bar{X} \ar[r]^{1}_{\iso}
      & J' X \ar[r]^{\compose{\alpha^{-1}_{X}}{Sf}}
      & S \U
  } \\[-\dp\strutbox] \end{gathered} \qedhere \]
\end{proof}

\begin{problem} \label{problem:weq.wkrelu.transfer}
  Given a square of functors as in Problem~\ref{problem:universe.transfer.concrete}, such that
  \begin{itemize}
  \item $J$ and $J'$ are fully faithful,
  \item $\D$ and $\D'$ are both univalent,
  \item $S$ is an equivalence, and
  \item $R$ is essentially surjective and full,
  \end{itemize}
  to construct an equivalance between weak $J$-relative universes and weak $J'$-relative universes.
\end{problem}

\begin{construction}[\coqident{weq_weak_relative_universe_transfer}]{problem:weq.wkrelu.transfer} \label{constr:weq.wkrelu.transfer}
  The equivalence $S$ of univalent categories induces an equivalence between morphisms in $\D$ and in $\D'$.
  Lemma~\ref{lemma:transfer_relwku_equiv} implies that this restricts to an equivalence between weak relative universes as desired.
\end{construction}

\section{CwF structures, representable maps of presheaves and the Rezk completion}\label{sec:cwf_rezk}

In this section, we show that CwF structures can be seen as relative universes, and hence apply the results of the previous section to transfer CwF structures along weak equivalences of categories: in particular, from a category to its Rezk completion.

We also consider \emph{representable maps of presheaves}, corresponding similarly to relative \emph{weak} universes, and use the results on relative universes to elucidate their relationship to CwF structures.

The resulting transfers and relationships are summed up in the following diagram, whose vertical maps all simply forget chosen structure:
\[
  \begin{xy}
   \xymatrix{
                 \cwf(\C) \ar[d] \ar[r]^{\simeq}
                 & \relu(\yofunctor[\C])  \ar[r] \ar[d]
                 & \relu(\yofunctor[\RC{\C}])\ar[r]^{\simeq} \ar[d]^{\simeq}
                 & \cwf(\RC{\C}) \ar[d]^{\simeq}
                 \\
                 \rep(\C) \ar[r]^{\simeq}
                 & \relwku(\yofunctor[\C]) \ar[r]^{\simeq}
                 & \relwku(\yofunctor[\RC{\C}])    \ar[r]^{\simeq}
                 &  \rep (\RC{\C})
   }
  \end{xy}
\]

\subsection{Representable maps of presheaves}

\begin{definition}[\coqident{rep_map}] \label{def:rep_map}
 A map $\p : \Tm \to \Ty$ of presheaves on a category $\C$ is \emph{representable} if for each for each $\Gamma : \C$ and $A : \Tymap{\Gamma}$, there exists some representation of the fiber of $\p$ over $A$, i.e.\ some $\deppr{A} : \compext{\Gamma}{A} \to \Gamma$ and $\te_A \in \Tmmap{\compext{\Gamma}{A}}$ satisfying the conditions of Definition~\ref{def:cwf.fiore}, item \ref{item:cwf.fiber-rep}.%
 \footnote{Note that in \cite{awodey_natural_published}, Awodey takes \emph{representable map} to mean what we call a CwF structure, i.e.\ to include choices of representations of the fibers; cf.\ \cite[\textsection 1.1, Algebraic character]{awodey_natural_published}.}

  In other words, a representable map of presheaves on $\C$ is just like a CwF structure, except that the representations are merely assumed to exist, not included as chosen data.
\end{definition}

Evidently, the underlying map $\p : \Tm \to \Ty$ of any CwF structure is representable, just by forgetting its chosen representations.
This can sometimes be reversed:

\begin{lemma}[\coqident{isweq_from_cwf_to_rep}]\label{lemma:map_cwf_rep_equiv}
  If $\C$ is a univalent category, then the forgetful map from CwF structures to representable maps of presheaves on $\C$ is an equivalence.
  That is, any representable map of presheaves on $\C$ carries a unique choice of representing data.
\end{lemma}

\begin{proof}
  If suffices to show that for a given map $\p : \Tm \to \Ty$, and for any $\Gamma : \C$ and $A : \Tymap{\Gamma}$, representing data $(\compext{\Gamma}{A},\deppr{A},\te_A)$ for the fiber is unique if it exists.

  Such data is always unique up to isomorphism, for any category $\C$, since pullbacks are unique up to isomorphism and $\yofunctor[\C]$ is full and faithful.
  But when $\C$ is univalent, this uniqueness up to isomorphism can be translated into literal uniqueness, as required.

  (Alternatively, following Constructions \ref{construction:equiv_cwf_relu} and \ref{construction:equiv_rep_relwku} below, we could see this result as essentially a special case of Corollary~\ref{cor:weq.relu.relwku} on universe structures.)
\end{proof}

\subsection{Transfer of CwF structures and representable maps of presheaves}

In order to apply the transfer results of the previous section, we first establish the equivalences between CwF structures on a category $\C$ (resp.\ representable maps of presheaves) and relative (weak) universes on $\yofunctor[\C]$.

\begin{problem}\label{problem:equiv_cwf_relu}
  Given a category $\C$, to construct an equivalence between $\cwf(\C)$ and $\relu(\yofunctor[\C])$.
\end{problem}

\begin{construction}[\coqident{weq_cwf_structure_RelUnivYo}]{problem:equiv_cwf_relu}\label{construction:equiv_cwf_relu}
  This is a matter of reassociating components, and replacing two quantifications over elements of a presheaf---once
  over $\Tymap{\Gamma}$, once over $\Tmmap{\compext{\Gamma}{A}}$---by quantification over the respective isomorphic sets
  of natural transformations into $\Ty$ and $\Tm$.
\end{construction}

\begin{problem}\label{problem:equiv_rep_relwku}
  Given a category $\C$, to construct an equivalence between $\rep(\C)$ and $\relwku(\yofunctor[\C])$.
\end{problem}

\begin{construction}[\coqident{weq_rep_map_weakRelUnivYo}]{problem:equiv_rep_relwku} \label{construction:equiv_rep_relwku}
  Similar to Construction~\ref{construction:equiv_cwf_relu}.
\end{construction}

Next, we make use of this to transfer CwF structures and representable maps along weak equivalences,
by viewing them as relative (weak) universes and applying the transfer results for those.

\begin{problem}\label{problem:transfer_cwf}
 Given a weak equivalence $F : \C \to \D$, where $\D$ is univalent, to construct a map $\cwf(\C) \to \cwf(\D)$.
\end{problem}

\begin{construction}[\coqident{transfer_cwf_weak_equivalence}]{problem:transfer_cwf}\label{construction:transfer_cwf}
  We construct a map $\relu(\yofunctor[\C]) \to \relu(\yofunctor[\D])$ as an instance of Construction~\ref{constr:universe.struct.on} and obtain the desired map by composition with the equivalence
  to CwF structures
  of Construction~\ref{construction:equiv_cwf_relu}.
  Consider the diagram
  \[
        \xymatrix@R=3pc@C=4pc{
                 \C \drtwocell<\omit>{\alpha} \ar[r]^-{\yofunctor[\C]} \ar[d]_{F}
                        & \preShv{\C} \ar@/_/[d]_{\ES}
                 \\ \D \ar[r]_-{\yofunctor[\D]}
                        & \preShv{\D} \ar@/_/[u]^{\simeq}_{F^{\circ}}
        }
  \]
  Here, the functor $F^{\circ}$ given by precomposition with $F\op$ is a weak equivalence between univalent categories, and hence a
  strong equivalence with inverse $\ES$.
  The isomorphism $\alpha$ is constructed as follows:
  Note that fully faithful functors reflect isomorphisms; we apply this for the functor $F^{\circ}$ of precomposition with $F\op$.
  It hence suffices to construct a natural isomorphism from  $\compose{\yofunctor[\C]}{\compose{\ES}{F^{\circ}}} \simeq \yofunctor[\C]$
     to $\compose{F}{\compose{\yofunctor[\D]}{F^{\circ}}}$.
   But this is an instance of a general isomorphism:
   indeed, for any functor $G : \A \to \X$, we have natural transformation
   from $\yofunctor[\A]$ to $\compose{G}{\compose{\yofunctor[\X]}{G^{\circ}}}$,
   and this natural transformation is an isomorphism when $G$ is fully faithful.
   This ends the construction of the natural isomorphism $\alpha$.
   The functor $\ES$ preserves pullbacks since it is fully faithful and essentially surjective.
   The hypotheses of Problem~\ref{problem:universe.struct.on} are easily checked,
   hence Construction~\ref{constr:universe.struct.on} applies.
\end{construction}

\begin{problem}\label{problem:transfer_rep}
 Given a weak equivalence $F : \C \to \D$, to construct an equivalence $\rep(\C) \simeq \rep(\D)$.
\end{problem}

\begin{construction}[\coqident{transfer_rep_map_weak_equivalence}]{problem:transfer_rep}\label{construction:transfer_rep}
   A direct instance of Construction \ref{constr:weq.wkrelu.transfer}.
\end{construction}

Putting everything together, we obtain:

\begin{problem}\label{problem:weq_cwf_reprc}
  For any category $\C$, to construct an equivalence between representable maps on $\C$ and CwF structures on $\RC{\C}$.
\end{problem}

\begin{construction}[\coqident{weq_rep_map_cwf_Rezk}]{problem:weq_cwf_reprc}
  Construction~\ref{construction:transfer_rep}, applied to $\eta_{\C}$, gives us an equivalence $\rep(\C) \simeq \rep(\RC{\C})$.
  On the other hand, Lemma~\ref{lemma:map_cwf_rep_equiv} tells us that $\cwf(\RC{\C}) \simeq \rep(\RC{\C})$.
  Composing the first of these with the inverse of the second yields the desired equivalence.
\end{construction}

\section{Guide to the accompanying formalization}\label{sec:formalization}

In this section we give some details on the formalization of the material presented here.

\subsection{Material formalized for this article}
All constructions and theorems of this work have been formalized in the proof assistant Coq,
over the \UniMath library of univalent mathematics \cite{UniMath,UniMath2015}.
We rely particularly heavily on \UniMath's category theory library.

Our formalization can be found at \url{https://github.com/UniMath/TypeTheory}.
The main library will continue development, so naming, organization, etc.\ may change from what is presented here.

However, the version described in the present article will remain permanently available under the tag \nolinkurl{2017-ALV1-arxiv},
and the file \nolinkurl{Articles/ALV_2017.v} will be maintained to keep the main results of this article available and locatable over future versions of the library.

For the reader interested in exploring the formalization, we recommend starting with that file, and following backwards to find the details of definitions and constructions.
A browsable version is available at \url{https://unimath.github.io/TypeTheory/coqdoc/master/TypeTheory.Articles.ALV_2017.html}.

The specific material of the present article amounts to about 3700 lines of code in the formalization.
Additionally, this development required formalizing a further c.~1500 lines of general background material (mostly on category theory) that had not previously been given.

\subsection{Type theory of the formalization}

There are some subtleties in how the type theory of the formalization relates to what we set out in Section~\ref{sec:background-type-theory}.

Firstly, the type theory of Coq includes many powerful features; like \UniMath itself, we deliberately avoid most of these, staying within the fragment described in \cite[\textsection 1]{UniMath2015}.

Secondly, functional extensionality and univalence are provided by \UniMath as axioms---that is, as assumptions added to the global context.
We use functional extensionality freely, but do not use univalence except where explicitly required, as noted in the text.

Thirdly, in order to acquire resizing principles (not otherwise available in Coq), \UniMath uses type-in-type, and hence is in principle inconsistent.
\UniMath itself is careful not to use any consequences of type-in-type except for the resulting resizing principles.
We are even more restricted: we do not make direct or essential use of the resizing principles.
We do depend on them indirectly, since propositional truncation is implemented in \UniMath using resizing principles.
However, our use of truncation is always via the interface corresponding to its standalone axiomatization, as assumed in Section~\ref{sec:background-type-theory}.

\section{Summary and future work}

The above sections complete the construction of the maps and equivalences of types promised in the introduction: in particular,
\begin{itemize}
\item equivalence between split type-category structures and CwF structures;
\item equivalence between CwF structures and universes relative to the Yoneda embedding, and similarly between representable maps and weak such relative universes;
\item transfer of CwF structures and representable maps to the Rezk completion;
\item equivalence between CwF structures on a category $\C$ and representable maps of presheaves on its Rezk completion.
\end{itemize}

There are several natural interesting directions for further work:

\begin{itemize}
\item Define \emph{categories} of all the various structures considered here; show that the comparison constructions given here are all moreover functorial, and that our equivalences of types underlie equivalences of categories.
  Besides its intrinsic interest, this would make more of our constructions meaningful and useful in the classical setting.

\item Extend these constructions to give comparisons with other categorical structures considered for similar purposes in the literature: Dybjer’s original CwF’s; Cartmell’s contextual categories/C-systems; comprehension categories; categories with display maps\ldots

\item Understand further the transfer of CwF structures along the Rezk completion construction: does the result enjoy an analogous universal property, making it the ‘free univalent category with families’ on a category $\C$?
\end{itemize}

Progress in some of these directions may be already found in our formalization, though not included in the present article.

\subsection*{Acknowledgements.}
  \thegrants

\bibliographystyle{alpha}
\bibliography{literature}

\end{document}